\documentclass{amsart}
\usepackage{mathrsfs,amssymb,amsthm,amsmath,verbatim}
\usepackage{tikz}
\usetikzlibrary{arrows,positioning,chains,matrix,scopes}

\title{Selective game versions of countable tightness with bounded
  finite selections}

\author[L. F. Aurichi]{Leandro F. Aurichi$^1$}
\thanks{$^1$ Partially supported by FAPESP (2013/05469-7 and 2015/25725) and by GNSAGA}
\address{Instituto de Ci\^encias Matem\'aticas e de Computa\c
c\~ao,
Universidade de S\~ao Paulo, Caixa Postal 668,
S\~ao Carlos, SP, 13560-970, Brazil}
\email{aurichi@icmc.usp.br}

\author[A. Bella]{Angelo Bella$^2$}
\thanks{$^2$ The research that led to the present paper was partially
supported by a grant of the group GNSAGA of INdAM}
\address{Department of Mathematics, University of Catania,
Citt\`a
  Universitaria, Viale A. Doria 6, 95125 Catania, Italy}
\email{bella@dmi.unict.it}

\author[R. R. Dias]{Rodrigo R. Dias$^3$}
\thanks{$^3$ Partially supported by FAPESP (2012/09214-0)}
\address{
Centro de Matem\'atica, Computa\c c\~ao e Cogni\c c\~ao,
Universidade Federal do ABC,
Avenida dos Estados, 5001,
Santo Andr\'e, SP, 09210-580, Brazil}
\email{rodrigo.dias@ufabc.edu.br}

\begin{document}

\begin{abstract}
For a topological space $X$ and a point $x \in X$, consider the
following game -- related to the property of $X$ being countably tight
at $x$.
In each inning $n\in\omega$,
the first player chooses a set $A_n$ that clusters at $x$,
and then the second player picks a point $a_n\in A_n$;
the second player is the winner if and only if 
$x\in\overline{\{a_n:n\in\omega\}}$.

In this work, we study variations of this game in which the second
player is allowed to choose finitely many points per inning rather
than one, but in which the number of points they are allowed to choose
in each inning has been fixed in advance.
Surprisingly, if the number of points allowed per inning is the same
throughout the play, then all of the games obtained in this fashion
are distinct.
We also show that a new game is obtained if the number of points
the second player is allowed to pick increases at each inning.

\end{abstract}

\maketitle

\newtheorem{defin}{Definition}[section]

\newtheorem{prop}[defin]{Proposition}

\newtheorem{prob}[defin]{Problem}

\newtheorem{lemma}[defin]{Lemma}

\newtheorem{corol}[defin]{Corollary}

\newtheorem{example}[defin]{Example}

\newtheorem{thm}[defin]{Theorem}

\newtheorem{question}[defin]{Question}

\newtheorem{fact}[defin]{Fact}

\theoremstyle{definition}

\newtheorem{remark}[defin]{Remark}

\newcommand{\pr}[1]{\left\langle #1 \right\rangle}
\newcommand{\mR}{\mathcal{R}}
\newcommand{\RR}{\mathbb{R}}
\newcommand{\QQ}{\mathbb{Q}}
\newcommand{\NN}{\mathbb{N}}
\newcommand{\ZZ}{\mathbb{Z}}
\newcommand{\mA}{\mathcal{A}}
\newcommand{\mV}{\mathcal{V}}
\newcommand{\mU}{\mathcal{V}}
\newcommand{\mP}{\mathcal{P}}
\newcommand{\mD}{\mathcal{D}}
\newcommand{\mB}{\mathcal{B}}
\newcommand{\C}{\mathrm{C}}
\newcommand{\mC}{\mathcal{C}}
\newcommand{\mO}{\mathcal{O}}
\newcommand{\mM}{\mathcal{M}}
\newcommand{\mF}{\mathcal{F}}
\newcommand{\D}{\mathrm{D}}
\newcommand{\0}{\mathrm{o}}
\newcommand{\OD}{\mathrm{OD}}
\newcommand{\Do}{\D_\mathrm{o}}
\newcommand{\sone}{\mathsf{S}_1}
\newcommand{\gone}{\mathsf{G}_1}
\newcommand{\gtwo}{\mathsf{G}_2}
\newcommand{\gn}[1]{\mathsf{G}_{#1}}
\newcommand{\sn}[1]{\mathsf{S}_{#1}}
\newcommand{\gfin}{\mathsf{G}_\mathrm{fin}}
\newcommand{\sfin}{\mathsf{S}_\mathrm{fin}}
\newcommand{\gf}{\mathsf{G}_f}
\newcommand{\Em}{\longrightarrow}
\newcommand{\menos}{{\setminus}}
\newcommand{\w}{{\omega}}
\newcommand{\1}{\textsc{One}}
\newcommand{\2}{\textsc{Two}}
\newcommand{\lev}{\mathrm{Lev}}
\newcommand{\cov}{\mathrm{cov}}
\newcommand{\fb}{\mathrm{fb}}
\newcommand{\dom}{\mathrm{dom}}
\newcommand{\id}{\mathrm{id}}

\section{Introduction}

Countable tightness is a classical topological property introduced in
\cite{tightness} as a generalization of first-countability:
a topological space $X$ is \emph{countably tight}
at a point $x\in X$ if every $A\subseteq X$ with $x\in\overline{A}$
has a countable subset $C$ with $x\in\overline{C}$.
In other words, $X$ is countably tight at $x$ if the countable sets in
$\Omega_p=\{A\subseteq X:x\in\overline{A}\}$ constitute a cofinal
family in $\Omega_p$ with respect to the order $\supseteq$.
Loosely speaking, in a countably tight space one can study clustering
properties -- hence the topology itself --
just by looking at the countable subsets of the space.

It is well-known that the product of two countably tight topological
spaces need not be countably tight (see e.g. \cite{arh72}), although
this will be the case if one of them is a first-countable space.
Thus, when dealing with countable tightness in topological products,
it is natural to consider the class of productively countably tight
spaces:
a topological space $X$ is \emph{productively countably tight} at
$x\in X$ if, for every topological space $Y$ that is countably tight
at a point $y\in Y$, the product space $X\times Y$ is countably tight
at the point $(x,y)$.
In \cite{arh79}, A. V. Arkhangel'ski\u\i{} obtained an internal
characterization for the productively countably tight Tychonoff
spaces (Theorem \ref{arh} below).
In \cite{gruen76} and, more recently, in \cite{aurichi-bella},
connections between productivity of countable tightness and selective
topological properties and their game versions have arisen (see
Theorem \ref{thm.1});
these will be the starting point of our investigations in this work.

A topological space $X$ has \emph{countable strong fan tightness} at a
point $x\in X$ \cite{sakai} if, for every sequence
$(A_n)_{n\in\omega}$ of elements of $\Omega_x$, we can select
$a_n\in A_n$ for $n\in\omega$ in such a way that
$\{a_n:n\in\omega\}\in\Omega_x$.
Countable strong fan tightness may be regarded
as a selective version of countable tightness,
in a combinatorial sense;
rather than just stating that every $A\in\Omega_x$ includes a
countable subset also in $\Omega_x$,
we now require that a new element of $\Omega_x$ can be obtained
by putting together points selected from
countably many elements of $\Omega_x$.
For reasons that will become clearer later,
we will adopt Scheepers's notation for selective topological
properties \cite{sch1} and refer to countable strong fan tightness at
$x$ as $\sone(\Omega_x,\Omega_x)$.

The property $\sone(\Omega_x,\Omega_x)$
has a combinatorial game of infinite length naturally associated to it
\cite{sch3}, which we denote by $\gone(\Omega_x,\Omega_x)$.
This game is played between the players \1 and \2
according to the following rules.
In each inning $n\in\omega$, \1 chooses a member $A_n\in\Omega_x$,
and then \2 picks a point $a_n\in A_n$.
The winner is \2 if $\{a_n : n\in\omega\} \in \Omega_x$,
and \1 otherwise.
If $P$ is a player of this game, we will write
$P \uparrow \gone(\Omega_x,\Omega_x)$ instead of ``$P$ has a winning
strategy in $\gone(\Omega_x,\Omega_x)$''
-- and this notational convention will be extended to all of the other
games we will consider in this work.

It is clear that, for every space $X$ and every $x\in X$,
$$
\2\uparrow\gone(\Omega_x,\Omega_x)\Longrightarrow
\1\not\,\uparrow\gone(\Omega_x,\Omega_x)\Longrightarrow
\sone(\Omega_x,\Omega_x).
$$
As it turns out --
by putting together Theorems 3.9 and 4.2 of
\cite{gruen76}\footnote{Although the game considered in \cite{gruen76}
appears to be quite different from $\gone(\Omega_x,\Omega_x)$, these
two games are closely related -- see Remark \ref{rmk}.}
and Theorem 2.7 of \cite{aurichi-bella} --,
the property of being productively countably tight
also lies between $\2\uparrow\gone(\Omega_x,\Omega_x)$ and 
$\sone(\Omega_x,\Omega_x)$, at least for Tychonoff spaces:

\begin{thm}[Gruenhage \cite{gruen76}, Aurichi--Bella \cite{aurichi-bella}]
\label{thm.1}
Let $X$ be a topological space and $x\in X$.
\begin{itemize}
\item[$(a)$]
If $\2\uparrow \gone(\Omega_x, \Omega_x)$ on $X$,
then $X$ is productively countably tight at $x$.
\item[$(b)$]
If $X$ is a Tychonoff space and
is productively countably tight at $x$,
then $\sone (\Omega_x, \Omega_x)$ holds in $X$.
\end{itemize}
\end{thm}

The issues we address in Section 2 refer to the properties mentioned
above.
Answering Question 2.9 of \cite{aurichi-bella}, 
we show (in Example \ref{ex.scheepers}) that
productivity of countable tightness does not imply
$\1\not\,\uparrow\gone(\Omega_x,\Omega_x)$.
The (rather unexpected) fact that the counterexample considered
has a different behaviour with respect to a seemingly minor
variation of the game $\gone(\Omega_x,\Omega_x)$
prompts us to investigate variations of this game in which \2 may
pick more than one point per inning.
This will be our main interest in Section 3, in which we show that,
surprisingly, each finite bound in the number of points \2 is allowed
to pick per inning leads to a different game.
We conclude in Section 4 with a discussion on possible new directions
that can be pursued from the results obtained in this work.

A word on notation.
The set of the natural numbers is denoted by $\omega$,
and we write $\NN$ in place of $\omega\setminus\{0\}$.
For a set $A$, the symbol ${}^{<\omega}A$
stands for the set of all of the finite sequences of
elements of $A$; furthermore, we write
${}^{\le\omega}A$ instead of $({}^{<\omega}A)\cup(\mbox{}^{\omega}A)$.
Finally,
\begin{itemize}
\item[$\cdot$]
$[A]^n$ denotes the set of all of the subsets of $A$ of cardinality
  $n$, for a given $n\in\omega$;
\item[$\cdot$]
$[A]^{<\aleph_0}$ denotes $\bigcup_{n\in\omega}[A]^n$;
\item[$\cdot$]
$[A]^{\aleph_0}$ denotes the set of all of the countable infinite
  subsets of $A$;
\item[$\cdot$]
$[A]^{\le\aleph_0}$ denotes $[A]^{<\aleph_0}\cup[A]^{\aleph_0}$.
\end{itemize}

All of the topological spaces we consider in this text are
assumed to be $T_1$.

\section{Some examples concerning $\sone(\Omega_x, \Omega_x)$,
$\gone(\Omega_x, \Omega_x)$ and productivity of countable tightness}

We begin this section with a couple of results witnessing
some differences between the properties considered in the
Introduction. From Theorem \ref{thm.1}, we know that, if $\2\uparrow
\gone(\Omega_x, \Omega_x)$ on some topological space $X$, then $X$ is
productively countably tight at $x$. An important class of spaces show
that this implication cannot be reversed: the one-point
compactifications of Mr\'owka spaces \cite{mrowka}.

\begin{prop}
\label{psi-space}
Let $\mathcal{A}$ be an uncountable almost disjoint family on $\omega$
and $\Psi(\mathcal A)$ be the corresponding Mr\'owka space. Let $X=\{p\} \,\dot\cup\, \Psi(\mathcal A)$ be the
one-point compactification of $\Psi(\mathcal A)$. Then $X$ is a
compact space that is productively countably tight at the point $p$,
yet $\2\not\,\uparrow \gone(\Omega_p,\Omega_p)$ on $X$.
\end{prop}

\begin{proof}
Since $X$ is a compact space of countable tightness,
it follows from Theorem 4 of \cite{malyhin} that
$X$ is productively countably tight at $p$.
Since $X$ is separable but not first-countable,
Theorem 3.6 of \cite{gruen76} tells us that
\2 does not have a winning strategy in
$\gone(\Omega_p,\Omega_p)$.
\end{proof}

We will see later on (in Example \ref{leandro}) that for a specific
case of the previous proposition we can even obtain $\1 \uparrow
\gone(\Omega_p, \Omega_p)$.

In view of the two chains of implications mentioned in the
Introduction, it is natural to ask the relation between being
productively countably tightness at a point $x$ and $\sone(\Omega_x,
\Omega_x)$ for Tychonoff spaces. A class of spaces shows that one of
the implications is not true: the spaces of continuous real-valued
functions defined on a first-countable uncountable $\gamma$ space
\cite[Theorem 6]{todorcevic}.

\begin{prop}
\label{ex.todorcevic}
Let $X = C_p(Y)$, where $Y$ is a first-countable uncountable $\gamma$
space. Then $\1\not\,\uparrow \gone(\Omega_x, \Omega_x)$ on $X$,
yet $X$ is not productively countably tight at $x$.
\end{prop}

\begin{proof}
It follows from results of Gerlits and Nagy \cite[Theorem 2]{g-n}
and Sharma \cite[Theorem 1]{sharma} that
$\1\not \uparrow\gone(\Omega_f,\Omega_f)$ for each $f\in C_p(Y)$;
however, $C_p(Y)$ is not productively countably tight since
the $G_\delta$-modification of $Y$ is not Lindel\"of
\cite[Theorem 1]{uspenskii}.
\end{proof}

The other implication was a question in a paper of the
first- and the second-named authors:

\begin{question}[Aurichi--Bella \cite{aurichi-bella}]
\label{q.1}
Let $X$ be productively countably tight at $x\in X$.
Is it true that $\1\not\,\uparrow\gone(\Omega_x,\Omega_x)$ on $X$?
\end{question}

In order to obtain a negative answer to Question \ref{q.1}
in a stronger sense,
we evoke an example from \cite{sch3}.

Here we consider the game $\gfin(\Omega_x, \Omega_x)$,
which is a standard variation of $\gone(\Omega_x, \Omega_x)$
-- see \cite{sch3}.
In each inning $n\in\omega$ of $\gfin(\Omega_x, \Omega_x)$,
\1 chooses $A_n\in\Omega_x$, and then
\2 chooses a finite subset $F_n\subseteq A_n$.
The winner is \2 if $\bigcup_{n\in\omega}F_n\in\Omega_x$,
and \1 otherwise.

\begin{example}
\label{ex.scheepers}
There exists a countable space $X$ with
only one non-isolated point $p$ that is productively countably
tight at $p$, but on which $\1\uparrow \gfin(\Omega_p,
\Omega_p)$.
\end{example}

\begin{proof}
Let $X = \omega \,\dot\cup\, \{p\}$ be the space in
\cite[pp.\! 250--251]{sch3}, defined as follows.
We first construct the function $F$ that will be
a winning strategy for \1 in $\gfin(\Omega_p,\Omega_p)$,
and then we make use of $F$ to define the topology of $X$.

Fix a partition $\{Y_n:n\in\omega\}$ of $\omega$
in infinite sets.
Now construct a strategy $F$ for \1 such that:
\begin{itemize}
\item[$(1)$]
for every sequence of non-empty finite sets $N_1, \ldots,
N_k\subseteq \omega$ that are subsets of different $Y_j$s,
there is an $i$ for which $F(N_1, \ldots, N_k)=Y_i$;
\item[$(2)$]
for every $i$, there is a unique sequence of finite sets $N_1,
\ldots, N_k$ such that:
\begin{itemize}
\item[$(a)$]
each $N_j$ is a non-empty finite subset of some $Y_{k_j}$ with
$k_j\ne i$;
\item[$(b)$]
if $j\ne m$, then $k_j\ne k_m$;
\item[$(c)$]
$F(N_1,\ldots,N_k)=Y_i$.
\end{itemize}
\end{itemize}

A way to define $F$ can be the following.
Let $\{T_n:n\in\omega\}$ be an injective enumeration of all finite
sequences of non-empty finite sets lying in different $Y_j$s,
and let $I_n$ be the set of those $j\in\omega$ such that $Y_j$
contains an element in the sequence $T_n$. By induction, for each
$n\in\omega$ let $i=\min(\omega\setminus (K\cup I_n))$,
where $K$ is the set of those $j\in\omega$ such that
$F(T_m)=Y_j$ for $m<n$. Then, put $F(T_n)=i$.

For each play $P=(O_1,N_1, O_2,N_2, \ldots)$ of
$\gfin(\Omega_p,\Omega_p)$ in $X$, we put
$S(P)=\bigcup_{k\in\omega}N_k$.

Every point of $\omega$ is isolated in $X$, while
a  local base at $p$ in $X$ consists of sets of the form
$V(G)=X\setminus \bigcup_{P\in G}S(P)$
for $G$ a finite set of
plays in which \1 uses the strategy $F$.

\begin{fact}[\cite{sch3}]
$\1\uparrow\gfin(\Omega_p, \Omega_p)$ on $X$.
\end{fact}

\begin{fact}
\label{prodtight}
$X$ is productively countably tight at $p$.
\end{fact}

In order to prove Fact \ref{prodtight},
we will make use of a result obtained in
Theorem 3.5 of \cite{arh79}.
Recall that a family $\mathcal{F}$ of subsets of a space $Y$ is a
\emph{$\pi$-network} at a point $y\in Y$ if every open neighbourhood
of $y$ includes some element of $\mathcal{F}$,
and that a family of sets is \emph{centred} if each of its nonempty
finite subfamilies has nonempty intersection.

\begin{thm}[Arkhangel'ski\u\i\mbox{} \cite{arh79}]
\label{arh}
Let $Y$ be a
Tychonoff  space and $y\in Y$. The following assertions are
equivalent:
\begin{itemize}
\item[$(a)$]
$Y$ is productively countably tight at $y$;
\item[$(b)$]
whenever $\bigcup_{i\in I}\mathcal{C}_i\subseteq [Y]^{\le
\aleph_0}$
is a $\pi$-network at $y$ such that each $\mathcal{C}_i$ is
nonempty
and centred,
there is a countable $I^*\subseteq I$  
such that $y\in \overline {\bigcup_{i\in I^*}A_i}$
for every choice of $A_i\in \mathcal{C}_i$, $i\in I^*$.
\end{itemize}
\end{thm}

We will show that condition $(b)$ holds for $X$ at $p$. Let
$\bigcup_{i\in I}\mathcal{C}_i\subseteq [X]^{\le \aleph_0}$ be a
$\pi$-network at $p$ such that each $\mathcal C_i$ is nonempty
and centred. Since
$X$ is a regular space, we can assume that each $\mathcal C_i$
consists of closed sets. We may further assume that $p\notin
\bigcap \mathcal C_i$ for every $i\in I$, since an $i'\in I$
satisfying $\forall A \in \mathcal{C}_{i'} \; (p\in A)$ 
would immediately imply the required condition with $I^*=\{i'\}$.

\begin{lemma}
\label{uniqdet}
Let $P$ be a play during which player $\1$ uses $F$.  If $S$ is
infinite
and $S\subseteq S(P)$,  then $P$ is uniquely determined.
\end{lemma}

\begin{proof} The key point is that, if $P$ and $P'$ are
distinct plays in which  $\1$ uses $F$, then $S(P)\cap S(P')$
must be
finite. Let $P=(O_1, N_1, O_2,N_2 \ldots)$ and
$P'=(O'_1,N'_1,O'_2,N'_2, \ldots)$, and let $k$ be the least
integer
for which $N_k\ne N'_k$.  Taking into account property $(2)$ of
$F$, we have $F(N_1, \ldots,N_k)\ne F(N'_1, \ldots,N'_k)$ and
consequently  $N_{k+1}\cap N'_{k+1}=\emptyset $.  Next, we  see
that the  sets $F(N_1,\ldots, N_k)$, $F(N_1, \ldots, N_{k+1})$,
$F(N'_1, \ldots, N'_k)$, $F(N'_1, \ldots, N'_{k+1})$ are mutually
distinct, and consequently   the  sets $N_{k+1}, N_{k+2},
N'_{k+1}, N'_{k+2}$ are pairwise disjoint.  By  continuing to
argue in this
manner, we  conclude that $S(P)\cap S(P')=N_1\cup\cdots \cup
N_{k-1}$.
\end{proof}

\begin{lemma}
\label{3}
For any open set $U\subseteq X$ with $p\in
U$ there exists a countable set $I'\subseteq I$ and pairwise
disjoint
elements $A_i\in \mathcal C_i$ satisfying $A_i\subseteq U$ for
$i\in
I'$.
\end{lemma}

\begin{proof}
This follows easily from $\bigcup_{i\in I}\mathcal{C}_i$  being
a $\pi$-network of $X$  consisting of closed sets, together with
the fact that $p\notin \bigcap \mathcal{C}_i$ for each $i\in I$.
\end{proof}

\begin{proof}[Proof of Fact \ref{prodtight}]
Using  Lemma \ref{3}, take a countable
set $I'_0\subset I$ and  pairwise disjoint elements $A_i\in
\mathcal
C_i$ for $i\in I'_0$. If  for any $B_i\in \mathcal C_i$, $i\in
I'_0$,
we have $p\in \overline {\bigcup_{i\in I'_0}B_i}$, then we
stop. Otherwise, since the families $\mathcal{C}_i$ are centred,
we
may assume that  $p\notin \overline {\bigcup_{i\in I'_0}A_i}$.
Then we fix a finite set $G_0 $ of plays in which  \1 uses the
strategy $F$   such that
$\bigcup_{i\in I'_0}A_i \cap V(G_0)=\emptyset $. Applying again
Lemma \ref{3} to the open set $V(G_0)$,    we find a countable
$I'_1\subseteq I$
and a pairwise  disjoint family
$\{A_i\in \mathcal C_i:i\in I'_1\}$
such that $A_i\subseteq V(G_0)$ for
each $i\in I'_1$. If $I'_1$ is good for our purpose, then we
stop. Otherwise, we fix a  finite set $G_1 $ of plays in which \1
uses $F$
such that $\bigcup_{i\in I'_1}A_i\cap  V(G_1)=\emptyset $. Of
course,  we may choose the set $G_1$ disjoint from $G_0$.
Then, we continue  by working in the open set $V(G_0\cup G_1)$.
If the process never stops,  at the end we put
$I^*=\bigcup_{n\in\omega}I'_n$. We claim that  the countable set
$I^*$  satisfies condition $(b)$ of Theorem \ref{arh}. To this
end,
take $B_i\in \mathcal{C}_i$ for
each $i\in I^*$ and choose  a point $s_i\in A_i\cap B_i$ for each
$i\in I^*$. It suffices to check that $p\in \overline{\{s_i:i\in
I^*\}}$.  By contradiction, assume that  there is a finite
set $H$ of plays in which \1 makes use of $F$   such that $\{s_i:i\in
I^*\}\subseteq
\bigcup_{P\in H}S(P)$. Since  the set
$S_n=\{s_i:i\in I'_n\}$ is infinite,  there exists some $P_n\in
H$ such that  the set $S_n\cap S(P_n)$  is
infinite. Again by Lemma \ref{uniqdet}, we must have $P_n\in
G_n$.
   But,  according to our construction, the sets in
$\{G_n:n\in\omega\}$ are pairwise disjoint and consequently    the
plays $P_n$ should be
 mutually distinct. This is obviously impossible and we are
done.
\end{proof}

Fact \ref{prodtight} finishes the proof of Example \ref{ex.scheepers}.
\end{proof}

One of the nicest classes of productively countably tight spaces
is the class of the compact spaces of countable tightness 
\cite[Theorem 4]{malyhin}.
One could expect to have a positive answer to
Question \ref{q.1} in this class.
But, again, this is not the case.

\begin{example}
\label{leandro}
There exists a compact Hausdorff space of countable tightness $X$ and
a
point
$p\in
X$ such that $\1\uparrow\gone(\Omega_p,\Omega_p)$.
\end{example}

\begin{proof}
This is essentially a particular case of Proposition \ref{psi-space}.
Endow the set $\mbox{}^{\le\omega}\omega$ with the topology in
which
every point of $\mbox{}^{<\omega}\omega$ is isolated and basic
neighbourhoods of $f\in\mbox{}^\omega \omega$ are of the form
$\{f\}\cup\{f\upharpoonright j:j\in\omega\setminus k\}$ for
$k\in\omega$.
Let $X=\{p\}\;\dot\cup\;\mbox{}^{\le\omega}\omega$ be the
one-point
compactification of this space.
Let \1's first move in $\gone(\Omega_p,\Omega_p)$ be $\{(k):k\in
\omega\}$ and, in general, if \2 picks a point
$s\in\mbox{}^{<\omega}\omega$, let \1's move in the next inning
be
$\{s^\frown (k):k\in \omega\}\in \Omega_p$.
It is clear that this is a winning strategy for \1 in
$\gone(\Omega_p,\Omega_p)$,
since all of \2's moves lie on a single
branch of the tree $\mbox{}^{<\omega}\omega$.
\end{proof}

After the last two examples, an obvious further question emerges:

\begin{prob}
\label{q.2}
Can a compact space  of countable
tightness $X$ have a point $x$ such that $\1\uparrow
\gfin(\Omega_x, \Omega_x)$?
\end{prob}

Notice that Example \ref{leandro} cannot be used to get a
positive
answer to Problem \ref{q.2}. Indeed, the next observation shows
that even in compact spaces the games $\gone(\Omega_p,\Omega_p)$
and
$\gfin(\Omega_p,\Omega_p)$ can have a very distinct behaviour
(we will see further differences between games of this kind
in the next section).

\begin{prop}
\label{gtwo}
If $X$ is the compact space in Example \ref{leandro}, then
$\2\uparrow \gtwo(\Omega_p, \Omega_p)$.
\end{prop}

\begin{proof}
  We will prove that  $\2$ can  reply to  $\1$ in such a
way that,  for each  $n$,
the set  of all the answers played by \2 in the first $n$ innings
includes a set $\{s_1, ..., s_n\}$   with the property that
no branch contains two elements of it.
Note that, in this way, $\2$ wins the game.

We will do so by induction. With no loss of generality, we
can assume that  \1 does not play a set containing  points
of $\w^{\w}$. If, in the first inning,   $\1$ plays $A_1$, then 
$\2$
chooses $\{s_1, s_2\} \subset A_1$ such that $s_1$ and $s_2$ are
not in
the same branch. Suppose that at the end of the $n$-th inning,
the set
of all answers of  $\2$ contains a set  $\{s_1, ..., s_n\}$ with
the prescribed property.  Let $A_{n + 1}$ be the move of  $\1$
in the inning $n+1$.  If there is a point in $A_{n + 1}$ that
lies in a
branch missing $\{s_1,..., s_n\}$, then  $\2$ chooses this point
together with some other one. In  the remaining case, Since $p$
is
in the closure of $A_{n+1}$,  there
is at least one $s_i$ and two incompatible  elements  $a_1, a_2
\in A_{n +1}$ such that   $s_i \subset a_1$ and $s_i\subset 
a_2$.   The
answer of  $\2$ in the $(n+1)$-th inning will be  just
$\{a_1,a_2\}$.
Observe that  every branch meets the set $\{s_j: j \neq i\} \cup
\{a_1, a_2\}$ in at most one point.
\end{proof}

In view of Theorem \ref{thm.1},  one may wonder what the 
real strength of the property
$\2\uparrow\gtwo(\Omega_x, \Omega_x)$ is.

\begin{prob}  \label{q.two}  Let $X$ be a space and $x\in X$.
Does $\2\uparrow \gtwo(\Omega_x, \Omega_x)$ imply that $X$ is
productively countably tight at $x$?
 \end{prob}

A  simple example of a countable space that is not productively
countably tight is $\omega\cup\{p\}\subseteq \beta\omega$, for 
an
arbitrary $p\in \beta\omega\setminus\omega$ \cite{aurichi-bella}.
Still in \cite{aurichi-bella}, it was further observed
that, if $p$ is a selective ultrafilter, then $\sone(\Omega_p,
\Omega_p)$ holds in $\omega\cup \{p\}$.  The possibility that
such a space
could provide a negative answer to Problem \ref{q.two} 
is ruled out by the next fact:

\begin{prop} [\cite{bella}, Proposition 3]  \label{double} Let
$X$
be a space and $x$ be a non-isolated point of $X$. If $\2\uparrow
\gfin(\Omega_x, \Omega_x)$, then  $x$ is in the
closure of two disjoint subsets of $X\setminus \{x\}$.
\end{prop}

Here we will show a little more:  there   are actually countably
many disjoint sets that have $x$ in their closures. This fact
will follow from the next more general result.

\begin{prop}
Let $X$ be a space on which $\2\uparrow\gfin(\Omega_x, \Omega_x)$
for a non-isolated point $x$.
Then there is an almost disjoint family of cardinality
$2^{\aleph_0}$ of elements of $\Omega_x$.
\end{prop}

\begin{proof}
Let $\{s_n : n \in \w \}$ be a one-to-one enumeration of
$2^{<\w}$
such that,  if $s_m \subseteq s_n$, then $m \leq n$.
Fix a winning strategy $\varphi$ for  $\2$ in $\gfin(\Omega_x,
\Omega_x)$.
Let $Q_{s_0} = X \setminus \{x\}$ and $A_{s_0} =
\varphi(Q_{s_0})$.
For each $n \in \w$,
define $Q_{s_n} = X \setminus (A_{s_0} \cup \cdots \cup A_{s_{n -
1}} \cup \{x\})$ and
$A_{s_n} = \varphi(Q_{s_n\upharpoonright 0}, \dots,
Q_{s_n\upharpoonright (\dom(s_n)-1)}, Q_{s_n})$.
Since $\varphi$ is a winning strategy,
the set
$B_g = \bigcup_{k \in \w} A_{g \upharpoonright k}$
is an element of $\Omega_x$ for each $g \in 2^\w$.
Note that, if $g , h\in 2^{\w}$ are distinct,
then $B_g \cap B_h = \bigcup_{j \le k_0} A_{g \upharpoonright j}
\in [X]^{<\aleph_0}$,
where $k_0 = \min\{k \in \w: g(k) \neq h(k)\}$.
\end{proof}

Recall that a space $Y$ is \emph{strongly Fr\' echet}
\cite{siwiec-sf} (or \emph{countably bisequential}
\cite{michael5})
at a point $y$
if for any decreasing sequence $(A_n)_{n\in \omega}$ of elements
of
$\Omega_y$  we may pick  points  $a_n\in A_n$ in such a way that
$(a_n)_{n\in\omega}$ is a sequence  converging  to $y$.

In general, a space productively countably tight at a point  $p$
need not be Fr\' echet at $p$, as Proposition \ref{psi-space} above
illustrates.

\begin{fact}
\label{stronglyfrechet}
The space $X$ in Example \ref{leandro} is strongly
Fr\' echet at $p$.
\end{fact}

\begin{proof}
It follows from (16)(b) in \cite{siwiec}
(see also \cite[Proposition 3]{arh-bella}) that
a regular space $Y$ is strongly Fr\' echet at a point $y$
if and only if
$Y$ is Fr\'echet at $y$ and
$\sone(\Omega_y,\Omega_y)$ holds in $Y$.
Thus, as the space $X$ is productively countably tight
and hence satisfies $\sone(\Omega_p,\Omega_p)$
\cite{aurichi-bella},
it suffices to check that $X$ is Fr\' echet at $p$.

Let then $A\subseteq\mbox{}^{\le\omega}\omega$ be such that
$p\in\overline{A}$. If $|A\cap\mbox{}^\omega \omega|\ge\aleph_0$,
then any injective function from $\omega$ into
$A\cap\mbox{}^\omega \omega$ is a sequence that converges to $p$.
We may then assume that $A\subseteq\mbox{}^{<\omega}\omega$.
Thus, in order to conclude that there is a sequence of
points of $A$ converging to $p$, it suffices to show that $A$
includes an infinite antichain,
since any injective function from $\omega$ onto
an antichain of $X$ is a sequence that converges to $p$.

Let
$Y=\{s\in A:\exists\, u,v\in A\;(s\subseteq u \;\&\; s\subseteq v
\;\&\; u\perp v)\}$.
We will consider two cases:

\underline{Case 1.}
$Y$ is finite.

It follows from the
definition of the set $Y$ that, for each $t\in A\setminus
Y$, there is $f_t\in\mbox{}^\omega\omega$ such that every $r\in
A$
with $t\subseteq r$ satisfies $r\subseteq f_t$. Note that the set
$A\setminus Y$ is the union of disjoint maximal chains. Since $Y$
is
finite and $p\in\overline A$, these maximal chains are infinitely
many. Therefore, by picking a point in each of these maximal
chains,
we get an infinite antichain.

\underline{Case 2.}
$Y$ is infinite.

If every chain included in $Y$ is finite, then
$\{\max_{\subseteq}(C):C$ is a maximal chain included in $Y\}$ is
an
infinite antichain. Let us then assume that $Y$ includes an
infinite
chain.
Thus, let $f\in\mbox{}^\omega\omega$ and
$D\in[\omega]^{\aleph_0}$ be
such that $\{f\upharpoonright j:j\in D\}\subseteq Y$. First let
$j_0=\min(D)$ and $s_0=f\upharpoonright j_0$. As $s_0\in Y$, we
can
pick $t_0\in A$ such that $s_0\subsetneqq t_0$ and $t_0\nsubseteq
f$. Now let $j_1=\min(D\setminus\mathrm{dom}(t_0))$ and
$s_1=f\upharpoonright j_1$, and pick $t_1\in A$ with
$s_1\subsetneqq
t_1$ and $t_1\nsubseteq f$. Let then
$j_2=\min(D\setminus\mathrm{dom}(t_1))$ and $s_2=f\upharpoonright
j_2$, and so forth. By proceeding in this fashion, we construct
an
infinite antichain $\{t_n:n\in\omega\}\subseteq A$.
\end{proof}

It is easy to check that a space $Y$ is strongly Fr\' echet at
$y$
if and only if for every family $\{A_n:n\in \omega\}\subseteq
\Omega_y$ we may pick points $a_n\in A_n$ in such a way that the
set $\{a_n:n\in \omega\}$ contains a subsequence converging to
$y$.  To require that the entire set $\{a_n:n\in \omega\}$ is a
sequence converging to $y$ is a much stronger  condition, which
is called \emph{strictly Fr\' echet at $y$} in \cite{g-n}.

Another interesting feature of the space $X$ in Example
\ref{leandro} is that it is strongly Fr\' echet at $p$, but
not strictly Fr\'echet at $p$.
This is an immediate consequence of the following corollary to
Theorem 1 of \cite{sharma} and Theorem 3.9 of \cite{gruen76}
-- see also Remark \ref{rmk} in the next section.

\begin{thm}[Galvin--Gruenhage--Sharma]
\label{thm.sharma}
Let $Y$ be a space and $y\in Y$.
If $Y$ is strictly Fr\'echet at $y$, then
$\1\not\,\uparrow\gone(\Omega_y, \Omega_y)$.
\end{thm}

It is worth remarking that the arguments concerning strong and
strict
Fr\'echetness of the space in Example \ref{leandro} apply equally
to
the space in Example \ref{ex.scheepers}.

Although Question \ref{q.1} has a negative answer even for
compact spaces of countable tightness, there is yet another
relevant class of spaces to consider.

\begin{prob}
\label{q.3}
Let $X$ be a space that is bisequential at a point $x$. Is it true
that $\1\not\,\uparrow\gone(\Omega_x,\Omega_x)$?
\end{prob}

Recall that a space $Y$ is \emph{bisequential} at a point $y$
\cite{michael5} if, for every filter base $\mathcal{F}$ on $Y$
that
accumulates in $y$, there is a countable filter base
$\mathcal{G}$
that converges to $y$ such that $F\cap G\neq\emptyset$ for all
$F\in\mathcal{F}$ and $G\in\mathcal{G}$.
In \cite{arh79}, it was shown that, if $Y$ is bisequential at
$y$,
then $Y$ is productively countably tight at $y$.

As an attempt to answer Problem \ref{q.3} in the negative, we
may
ask:

\begin{prob}
\label{bi?}
Is the space $X$ from Example \ref{leandro} bisequential at
$p$?
\end{prob}

\section{The differences between various games}

In this section, we concentrate on investigating the differences
that show up in several selective games related to countable
tightness
according to the number of points that player $\2$ is allowed to
select in each inning.

We have already seen (in Proposition \ref{gtwo}) that we obtain
a different game if $\2$ is allowed to pick two points per
inning instead of one.
Now we will extend this result not only to an arbitrary fixed quantity
of points, but also for a quantity that varies on the number of the
inning being played.

We first recall the following result from \cite[Section 3]{salvador}.
Here, for a function $f\in{}^{\omega}\NN$,
the notation $\sn{f}(\Omega_x,\Omega_x)$ stands for the following
property:
for every sequence
$(A_n)_{n\in\omega}$ of elements of $\Omega_x$, we can pick subsets
$F_n\subseteq A_n$ for $n\in\omega$ in such a way that
$|F_n|\le f(n)$ for each $n\in\omega$ and
$\bigcup_{n\in\omega}F_n\in\Omega_x$.
This is a natural generalization of $\sone(\Omega_x,\Omega_x)$,
hence a (formally) more general selective version of countable
tightness in combinatorial sense:
rather than picking one point out of each $A_n$, we can now select up
to $f(n)$ many points and, joining all of those points together
(for all $n\in\omega$), we must assemble a new element of the
family $\Omega_x$. The following result, together with Example
\ref{ex.salvador}, basically shows that, for an unbounded function
$f\in{}^{\omega}\NN$, $\sn{f}(\Omega_x, \Omega_x)$ is indeed a new
property --- but only one new property is given in this fashion, since
it does not matter which unbounded function is taken.

\begin{prop}[Garc\'ia-Ferreira--Tamariz-Mascar\'ua
\cite{salvador}]
\label{equiv}
Let $X$ be a space and $x\in X$.
\begin{itemize}
\item[$(a)$]
If $f\in{}^{\omega}\NN$ is bounded, then
$\sn{f}(\Omega_x,\Omega_x)$ is equivalent to
$\sone(\Omega_x,\Omega_x)$.
\item[$(b)$]
If $f,g\in{}^{\omega}\NN$ are both unbounded, then
$\sn{f}(\Omega_x,\Omega_x)$ is equivalent to
$\sn{g}(\Omega_x,\Omega_x)$.
\end{itemize}
\end{prop}

We can now consider, for a given $f\in{}^{\omega}\NN$, the game
$\gn{f}(\Omega_x,\Omega_x)$ naturally associated to the selective
property $\sn{f}(\Omega_x,\Omega_x)$.
In each inning $n\in\omega$ of this game, \1 chooses $A_n\in\Omega_x$,
and then \2 selects $F_n\subseteq A_n$ with $|F_n|\le f(n)$.
The winner is \2 if $\bigcup_{n\in\omega}F_n\in\Omega_x$,
and \1 otherwise.

For $k\in\NN$, we will write $\gn{k}(\Omega_x,\Omega_x)$ instead of
$\gn{\bar{f}_k}(\Omega_x,\Omega_x)$, where
$\bar{f}_k\in{}^{\omega}\NN$ is the constant function having range
$\{k\}$.

\begin{corol}  \label{gk->s1}  Let $X$ be a space  and $x\in X$.
  If  $\2\uparrow\gn{k}(\Omega_x,\Omega_x)$ on $X$ for some $k \geq
  1$, then $\sone(\Omega_x, \Omega_x)$ holds.
\end{corol}

The above corollary shows that, in general, $\2\uparrow
\gfin(\Omega_x, \Omega_x)$ is strictly weaker than $\2\uparrow
\gn{k}(\Omega_x, \Omega_x)$.  To see this, it suffices to observe
that  $\2\uparrow \gfin(\Omega_{\mathbf{0}},
\Omega_{\mathbf{0}})$ on $C_p(\RR)$
(see \cite[Theorem 3.6]{barman-dow} for a direct proof),
but $\sone(\Omega_{\mathbf{0}}, \Omega_{\mathbf{0}})$ fails for
$C_p(\RR)$ \cite[Theorem 1]{sakai}.
$C_p(\RR)$ also shows that $\2\uparrow \gfin(\Omega_x, \Omega_x)$
does not imply productive countable
tightness at $x$ (see \cite[Theorem 1]{uspenskii}).

The next result, which is a game version of Proposition
\ref{equiv},
is a first step towards drawing another line between the games
$\gn{k}(\Omega_x,\Omega_x)$ and $\gfin(\Omega_x,\Omega_x)$.

\begin{prop}
\label{equiv-games}
Let $X$ be a space and $x\in X$.
\begin{itemize}
\item[$(a)$]
If $f\in{}^{\omega}\NN$ is bounded, then the games
$\gn{f}(\Omega_x,\Omega_x)$ and $\gn{k}(\Omega_x,\Omega_x)$ are
equivalent, where $k=\lim\sup_{n\in\omega}f(n)\in\NN$.
\item[$(b)$]
If $f,g\in{}^{\omega}\NN$ are both unbounded, then the games
$\gn{f}(\Omega_x,\Omega_x)$ and $\gn{g}(\Omega_x,\Omega_x)$ are
equivalent.
\end{itemize}
\end{prop}

\begin{proof}

For $(a)$, we first note that a winning strategy for \2 in
$\gn{f}(\Omega_x,\Omega_x)$ is itself a winning strategy for \2 in
$\gn{k}(\Omega_x,\Omega_x)$, if \2 just ignores the finitely many
innings $n \in \omega$ in which $f(n)>k$.
For the converse, suppose that \2 has a winning strategy $\varphi$ in
$\gn{k}(\Omega_x,\Omega_x)$.
As the set $N = \{ n\in\omega : f(n)=k \}$ is infinite,
\2 can win an arbitrary play of $\gn{f}(\Omega_x,\Omega_x)$ by
skipping the innings $n \in \omega \setminus N$ and making
use of $\varphi$ in the innings $n \in N$
(considering, for the history of the play of
$\gn{k}(\Omega_x,\Omega_x)$ in which \2 applies $\varphi$,
only the innings that are in $N$).

We now deal with  \1.
Again, a winning strategy $\varphi$ for \1 in the game
$\gn{k}(\Omega_x,\Omega_x)$ is a winning strategy for \1 in
$\gn{f}(\Omega_x,\Omega_x)$, for \1 can pretend that the only valid
innings are the (cofinitely many) ones in
$\{ n\in\omega : f(n) \le k \}$
and making use of $\varphi$ in those innings.
Conversely, suppose that \1 has a winning strategy $\varphi$ in
$\gn{f}(\Omega_x,\Omega_x)$. Let $\{ n_i : i\in\omega \}$ be an
increasing enumeration of $\{ n\in\omega : f(n)=k \}$.
Now define a strategy for \1 in
$\gn{k}(\Omega_x,\Omega_x)$ as follows: in the inning $i\in\omega$, if
the play so far is $(A_0,F_0,\dots,A_{i-1},F_{i-1})$, \1's move is
$\varphi(G_0,G_1,\dots,G_{n_i-1})$, where $G_{n_j}=F_j$ for each $j<i$
and $G_m=\emptyset$ for all other values of $m$.
Since $\varphi$ is a winning strategy for \1 in
$\gn{f}(\Omega_x,\Omega_x)$, we have
$\bigcup_{i\in\omega}F_i=\bigcup_{m\in\omega}G_m\notin\Omega_x$;
therefore, this defines a winning strategy for \1 in
$\gn{k}(\Omega_x,\Omega_x)$.

For $(b)$, we must show that
\begin{itemize}
\item[$\cdot$]
$\2 \uparrow \gn{f}(\Omega_x,\Omega_x)$ implies
$\2 \uparrow \gn{g}(\Omega_x,\Omega_x)$;
\item[$\cdot$]
$\1 \uparrow \gn{f}(\Omega_x,\Omega_x)$ implies
$\1 \uparrow \gn{g}(\Omega_x,\Omega_x)$.
\end{itemize}

For the first implication, let $\varphi$ be a winning strategy for \2
in $\gn{f}(\Omega_x,\Omega_x)$, and let $(n_i)_{i\in\omega}$ be an
increasing sequence in $\omega$ such that, for each $i\in\omega$, we
have $g(n_i)\ge f(i)$.
Then \2 can produce a winning strategy in $\gn{g}(\Omega_x,\Omega_x)$
by playing along the innings in $\{n_i:i\in\omega\}$ only, making use
of $\varphi$ in those innings and ignoring the other innings.

Finally, for the second implication, let $\varphi$ be a winning
strategy for \1 in $\gn{f}(\Omega_x,\Omega_x)$.
Let $(m_i)_{i\in\omega}$ be an increasing sequence in $\omega$
such that, for each $i\in\omega$, we have $f(m_i)\ge g(i)$.
We can now define a strategy $\psi$ for \1 in
$\gn{g}(\Omega_x,\Omega_x)$
by setting
$\psi(F_0,F_1,\dots,F_{i-1})=\varphi(G_0,G_1,\dots,G_{m_i-1})$,
where $G_{m_j}=F_j$ for each $j<i$
and $G_n=\emptyset$ for $n\notin\{m_j:j<i\}$.
As in the last part of $(a)$, it follows that
$\bigcup_{i\in\omega}F_i=\bigcup_{n\in\omega}G_n\notin\Omega_x$,
whence $\psi$ is a winning strategy.
\end{proof}

With Propositions \ref{equiv} and \ref{equiv-games} in mind,
we henceforth adopt the following convention:
whenever we write $\sn{f}$ and $\gn{f}$, it should be understood that
$f\in {}^{\omega}\NN$ is unbounded.

Now we present all the variations in which we will be interested here.
It is immediate that, for every space $X$,
$p \in X$ and $k \in \NN$,
we have the following chain of implications:
$$\2 \uparrow \gn{k}(\Omega_p, \Omega_p) \Rightarrow
\2 \uparrow \gn{k + 1}(\Omega_p, \Omega_p) \Rightarrow
\2 \uparrow \gn{f}(\Omega_p, \Omega_p) \Rightarrow
\2 \uparrow \gfin(\Omega_p, \Omega_p)$$
In what follows, we will show that none of these
implications can be reversed in general.

\begin{example}
\label{CpR}
There is a topological space showing that
$\2\uparrow\gfin(\Omega_x, \Omega_x)$
does not imply
$\sn{f}(\Omega_x, \Omega_x)$.
\end{example}

\begin{proof}
This is witnessed by $C_p(\mathbb{R})$.
We have already remarked, right after Corollary \ref{gk->s1}, that
\2 $\uparrow\gfin(\Omega_{\mathbf{0}},\Omega_{\mathbf{0}})$ on $C_p(\mathbb{R})$.
On the other hand,
$C_p(\mathbb{R})$ does not satisfy
$\sn{f}(\Omega_{\mathbf{0}},\Omega_{\mathbf{0}})$
by Theorem 3.13 of \cite{salvador}.
\end{proof}

Another space satisfying the conditions in Example \ref{CpR}
can be found in Example 3.8 of \cite{salvador}.

\begin{example}\label{ex.salvador}
There is a space showing that
$\2\uparrow\gn{f}(\Omega_x, \Omega_x)$
does not imply
$\sn{1}(\Omega_x, \Omega_x)$.
\end{example}

\begin{proof}
This is witnessed by the following space,
described in Example 3.7 of \cite{salvador}.
Consider, on $X=(\omega\times\omega)\;\dot\cup\;\{p\}$, the topology
in which points of $\omega\times\omega$ are isolated and basic
neighbourhoods of $p$ are of the form
$V_H=X\setminus\bigcup_{h\in H}\{(n,h(n)):n\in\omega\}$, for $H$ a
finite subset of ${}^{\omega}\omega$.
It is clear that $X$ does not satisfy $\sn{1}(\Omega_p, \Omega_p)$,
since each set $C_n=\{(n,m):m\in\omega\}$ is in $\Omega_p$.
Note that a subset $A$ of $\omega\times\omega$ satisfies
$p\in\overline{A}$ if and only if
$\sup\{|A\cap C_n|:n\in\omega\}=\aleph_0$;
thus, we can obtain a winning Markov strategy for \2 in the game
$\gn{f}(\Omega_p,\Omega_p)$ on $X$ as follows:
in each inning $j\in\omega$, if $A_j\subseteq\omega\times\omega$ is
the set played by \1, let $n_j\in\omega$ be such that
$|A_j\cap C_{n_j}|\ge f(j)$, and then declare \2's move to be
$F_j\subseteq A_j\cap C_{n_j}$ with $|F_j|=f(j)$.
\end{proof}

We have seen in Proposition \ref{equiv} that the properties
$\sn{k}(\Omega_x, \Omega_x)$ for $k \in \NN$ are all equivalent.
We also have seen in Example \ref{leandro} that this is not the case
between $\gone(\Omega_x, \Omega_x)$ and $\gn{2}(\Omega_x, \Omega_x)$.
Now we will see that all of the games $\gn{k}(\Omega_x, \Omega_x)$
for $k\in\NN$ are distinct.

\begin{example}\label{ex.k.k1}
For each $k \in \NN$, there is a countable space $X_k$
with only one non-isolated point $p$ on which
$\1\uparrow\gn{k}(\Omega_p, \Omega_p)$ and
$\2\uparrow\gn{k + 1}(\Omega_p, \Omega_p)$.
\end{example}

\begin{proof}
This space is a variation of the space from Example \ref{ex.scheepers}.
Write $\omega=\dot\bigcup_{s\in{}^{<\omega}\omega}N_s$
with each $N_s$ infinite.
For each $s\in{}^{<\omega}\omega$,
fix a bijective enumeration $[N_s]^k = \{K^s_i:i\in\omega\}$.
Now consider, on $X_k=\omega\;\dot\cup\;\{p\}$,
the topology in which points in $\omega$ are isolated and
basic neighbourhoods of $p$ are of the form
$$
X_k\setminus\bigcup_{g\in\mathcal{G}}\bigcup_{j\in\omega}
K^{g\upharpoonright j}_{g(j)}
$$
for $\mathcal{G}\subseteq{}^{\omega}\omega$ a finite set.

It is clear that $\1$ has a winning strategy in the game
$\gn{k}(\Omega_p, \Omega_p)$ on $X_k$.
We will now see that $\2$ has a winning strategy in the game
$\gn{k+1}(\Omega_p, \Omega_p)$ on $X_k$.

\begin{defin}

A \emph{fat branch} is a subset of $X_k$ of the form
$\bigcup_{j\in\omega} K^{g\upharpoonright j}_{g(j)}$
for $g\in{}^{\omega}\omega$.

\end{defin}

The following is a simple but useful fact.

\begin{lemma}
\label{lemma.2}

Let $Y\subseteq\omega$ be such that $p\in\overline{Y}$ in $X_k$.
Then there is $B\subseteq Y$ with $|B| = k+1$
that cannot be covered with a single fat branch.

\end{lemma}

\begin{proof}

Suppose, by way of contradiction, that
\begin{equation}\tag{$\dagger$}
\textrm{each }
B\subseteq Y
\textrm{ with }
|B| = k+1
\textrm{ is included in a fat branch.}
\end{equation}
This implies, in particular, that the set
$S=\{s\in{}^{<\omega}\omega:Y\cap N_s\neq\emptyset\}$
is a chain in $\mbox{}^{<\omega}\omega$:
if $s,t\in S$ were incompatible,
we would contradict $(\dagger)$
by picking $m\in Y\cap N_s$ and
$n\in Y\cap N_t$
and then considering a set $B\subseteq Y$ with $|B| = k+1$ satisfying
$\{m,n\}\subseteq B$.

\emph{Case 1.}
$S$ is infinite.

Let $g=\bigcup S\in\mbox{}^{\omega}\omega$.
It follows from our hypothesis that
$Y\nsubseteq\bigcup_{j\in\omega}K^{g\upharpoonright j}_{g(j)}$.
Pick $n\in
Y\setminus\bigcup_{j\in\omega}K^{g\upharpoonright j}_{g(j)}$, and
let $s\in S$ be such that $n\in Y\cap N_s$.
Now pick $t\in S$ with $s\subsetneqq t$ and choose
$m\in Y\cap N_t$.
Then
we need two distinct fat branches in order to cover the set $\{m,n\}$,
since
$n\notin
K^{g\upharpoonright\dom(s)}_{g(\dom(s))}=K^{s}_{t(\dom(s))}$.
Again, this contradicts $(\dagger)$.

\emph{Case 2.}
$S$ is finite.

Let $u=\max S$. By the same reasoning applied in Case 1, for each
$s\in S$ with $s\subsetneqq u$ we must have $Y\cap N_s\subseteq
K^{u\upharpoonright\dom(s)}_{u(\dom(s))}$.
It cannot be the case that $|Y\cap N_u|\ge k+1$, since
a set $B\subseteq Y\cap N_u$ with $|B|=k+1$
would contradict $(\dagger)$.
Thus, there is $i\in\omega$ such that $Y\cap N_u\subseteq K^u_i$.
Now let $g\in\mbox{}^{\omega}\omega$ be such that $u\subseteq g$ and
$g(\dom(u))=i$.
Then
$Y\subseteq\bigcup_{j\in\omega}K^{g\upharpoonright j}_{g(j)}$, which
contradicts the hypothesis on $Y$.
\end{proof}

\begin{lemma}
\label{EUB}

\2 can play the game $\gn{k+1}(\Omega_p, \Omega_p)$ on $X_k$
in such a way that, after the inning $n\in\omega$,
the set of all of the points picked by \2 includes a subset $E_n$
with $|E_n|\ge n+k+1$
such that no fat branch contains more than $k$ points of $E_n$.

\end{lemma}

\begin{proof}

For $Y\subseteq\omega$, define
$Y^\downarrow=
\bigcup_{n\in Y}\bigcup_{j\in\dom(s_n)}
K^{s_n\upharpoonright j}_{s_n(j)}$,
where $s_n\in{}^{<\omega}\omega$ is such that $n\in N_{s_n}$.

We proceed by induction on $n$.
For the initial inning, let \2's move be the set $E_0$
given by Lemma \ref{lemma.2}.
Now suppose that, after the inning $n\in\omega$, the set of all of the
points picked by \2 includes a set $E_n$ with $|E_n|\ge n+k+1$ such that
each fat branch contains at most $k$ points of $E_n$,
and let $A_{n+1} \in \Omega_p$.

Let
$L = \{i \in \w:$ there is a fat branch containing both $i$ and some
element of $E_n\}$.
If $A_{n+1} \setminus L \in \Omega_p$,
let \2's move be $\{i\}$ for some $i\in A_{n+1} \setminus L$,
and declare $E_{n+1}=E_n\cup\{i\}$.
Let us then assume that $A_{n+1} \setminus L \notin \Omega_p$.
It follows that $A_{n+1} \cap L \in \Omega_p$;
furthermore, as $E_n {}^{\downarrow}$ is finite,
we may also assume that
$A_{n+1}\cap E_n {}^{\downarrow} =\emptyset$.

For each $D\subseteq E_n$ with $|D|\le k$, let
$Z_D = \{i \in \w:$ there is a fat branch that includes $\{i\}\cup D\}$.
Now let $d\le k$ be maximal such that,
for some $D\subseteq E_n$ with $|D|=d$, the set
$$
A^*=
\left( A_{n+1} \setminus \bigcup_{j=d+1}^{k} \;
\bigcup_{C\in[E_n]^j} Z_C \right)
\cap Z_D
$$
is an element of $\Omega_p$.
Note that such a $d$ must exist:
if there is no $D\subseteq E_n$ with $|D|=k$ such that
$A_{n+1} \cap Z_D \in \Omega_p$, then
$A_{n+1} \setminus \bigcup_{C\in[E_n]^{k}} Z_C$ must be in $\Omega_p$;
now, if there is no $D\subseteq E_n$ with $|D|=k-1$ such that
$(A_{n+1} \setminus \bigcup_{C\in[E_n]^{k}} Z_C) \cap Z_D \in \Omega_p$, then
$(A_{n+1} \setminus \bigcup_{C\in[E_n]^{k}} Z_C) \setminus \bigcup_{C\in[E_n]^{k-1}} Z_C$
must be in $\Omega_p$; and so forth.
This process must stop at some $d>0$ since
$A_{n+1} \cap L \in \Omega_p$.

Let $B$ be the set obtained by applying Lemma \ref{lemma.2}
to the set $A^*$.
We will show that, by letting \2's move in this inning be $B$,
the set $E_{n+1}=(E_n\setminus D)\cup B$ will satisfy the condition
required by the induction.

Let $G\subseteq \omega$ be a fat branch.
If $G\cap (E_n \setminus D) =\emptyset$, then
$G\cap E_{n+1} = G\cap B$ has no more than $k$ points.
If $G\cap B=\emptyset$, then the set $G\cap E_{n+1}$ is a subset of
$G\cap E_n$, hence it has no more than $k$ points.
Let us now deal with the case in which the sets
$G\cap(E_n\setminus D)$ and $G\cap B$ are both non-empty.

Let $i\in G\cap (E_n \setminus D)$ and $m\in G\cap B$ be arbitrary.
Since $A_{n+1}\cap E_n {}^{\downarrow} =\emptyset$,
it follows that $B \cap \{i\}^{\downarrow} =\emptyset$;
thus, as $m\in G\cap B$ and $i\in G$, it must be the case that
$i\in N_t$ for some $t\subseteq s$, where $s\in {}^{<\omega}\omega$ is
such that $m\in N_s$.
Let $g\in {}^{\omega}\omega$ be such that
$G=\bigcup_{j\in\omega} K^{g\upharpoonright j}_{g(j)}$, and let
$h\in {}^{\omega}\omega$ be such that the fat branch
$H=\bigcup_{j\in\omega} K^{h\upharpoonright j}_{h(j)}$
includes the set $\{m\}\cup D$.
As $m\in G\cap H$, it follows that $g\cap h \supseteq s$.
Let $j_0=\dom(t)$.

We claim that $t=s$.
Suppose, to the contrary, that $t\subsetneqq s$.
Then $i\in K^t_{g(j_0)} = K^t_{h(j_0)} \subseteq H$, which implies
that $H$ is a fat branch that includes $C=\{i\}\cup D$.
But then $m\in Z_C$, since $m$ is also an element of $H$.
This contradicts the fact that $m\in B\subseteq A^*$.

Thus, as $i$ and $m$ are arbitrary, we have proved that
$(G\cap(E_n\setminus D)) \cup (G\cap B) \subseteq N_t$.
This implies that $G\cap E_{n+1} \subseteq G\cap N_t = K^t_{g(j_0)}$,
which in turn yields $|G\cap E_{n+1}|\le k$, as required.
\end{proof}

The existence of a winning strategy for \2
in the game $\gn{k+1}(\Omega_p, \Omega_p)$ on $X_k$
is an immediate consequence of Lemma \ref{EUB}.
\end{proof}

We will now see a space in which all of the games considered in this
work are undetermined.

\begin{example}\label{ex.gruen}
There is a countable space with only one non-isolated point $p$
on which $\1\not\,\uparrow\gone(\Omega_p,\Omega_p)$ and
$\2\not\,\uparrow\gfin(\Omega_p,\Omega_p)$.
\end{example}

\begin{proof}

This is essentially the space from Example 2.11 of \cite{gruen06}.

Consider, on the set $T={}^{\le\omega}2$,
the topology generated by the base
$\{ \{s\} : s\in{}^{<\omega}2\} \cup \{V_s : s\in{}^{<\omega}2\}$,
where $V_s=\{t\in T : s\subseteq t\}$ for each $s\in{}^{<\omega}2$.
The subspace ${}^{\omega}2$ of $T$ is homeomorphic to the Cantor set;
let then $B\subseteq {}^{\omega}2$ be a Bernstein set
in ${}^{\omega}2$
(see e.g. \cite[Theorem 11.4]{jw1}).

We will now consider the space $X_B = ({}^{<\omega}2) \,\dot\cup\, \{p\}$
in which every $s\in{}^{<\omega}2$ is isolated and basic neighbourhoods
of $p$ are of the form
$U_F = X_B\setminus\bigcup_{f\in F}\{f\upharpoonright j:j\in\omega\}$
for $F\in[B]^{<\aleph_0}$.

\begin{remark}
\label{rmk}
By essentially the same argument from Theorem 1 of \cite{galvin},
\2 (resp. \1) has a winning strategy in the game
$\gone(\Omega_x, \Omega_x)$ played on $X$ if and only if
\1 (resp. \2) has a winning strategy in the game
$\mathsf{G}^c_{O,P}(X,x)$,
defined as follows.
In each inning $n\in\omega$, \1 chooses an open neighbourhood
$V_n$ of
$x$ in $X$, and then \2 picks a point $x_n\in V_n$.
The winner is \1 if $\{x_n:n\in\omega\}\in\Omega_x$, and \2 otherwise.
\end{remark}

\begin{prop}[Gruenhage \cite{gruen06}]
\label{1.not.g1}
$\2\not\,\uparrow\mathsf{G}^c_{O,P}(X_B,p)$.
\end{prop}

We should point out that, although the space $X_B$ in our construction
is not quite the same as the one exhibited in Example 2.11 of
\cite{gruen06}, adapting the proof in Gruenhage's paper to the space
$X_B$ is quite straightforward.

By Remark \ref{rmk}, Theorem \ref{1.not.g1} is equivalent to
the statement that $\1\not\,\uparrow\gone(\Omega_p,\Omega_p)$ on $X_B$.
Thus, in order to conclude that all of the games we are considering
are undetermined on $X_B$ at $p$, we must show that 
$\2\not\,\uparrow\gfin(\Omega_p,\Omega_p)$ on $X_B$.
In order to prove this, we will need a few auxiliary results.

\begin{lemma}
\label{closure.XB}

Let $A\subseteq{}^{<\omega}2$.
The following statements are equivalent:
\begin{itemize}
\item[$(a)$]
$p\in\overline{A}$ in $X_B$;
\item[$(b)$]
$A$ includes an infinite antichain or
there is $g\in({}^{\omega}2)\setminus B$ such that the set
$A\cap\{g\upharpoonright j:j\in\omega\}$ is infinite.
\end{itemize}

\end{lemma}

\begin{proof}

The implication $(b)\Rightarrow(a)$ is clear.
For the converse, suppose that $p\in\overline{A}$ in $X_B$
and that every antichain included in $A$ is finite.
Then the same holds for the set
$A^-=\bigcup_{s\in A}\{s\upharpoonright j:j\le\dom(s)\}\supseteq A$.
Since $A^-$ is an infinite subtree of ${}^{<\omega}2$,
it must have at least one infinite branch.

{\bf Claim.}
The set
$C=\{f\in{}^{\omega}2:\forall j\in\omega\;(f\upharpoonright j\in A^-)\}$
is finite.

For each $f\in C$, there must be some $j_f\in\omega$ such that
$\{s\in A^-:s\supseteq f\upharpoonright j_f\}=
\{f\upharpoonright j':j_f\le j'\in\omega\}$,
for otherwise $A^-$ would include an infinite antichain.
This implies that $\{f\upharpoonright j_f:f\in C\}\subseteq A^-$
is an antichain, whence $C$ must be finite.

Thus, as $U_{C\cap B}$
is an open neighbourhood of $p$ in $X_B$,
it follows that the set
$D=A^-\cap U_{C\cap B}$ is such that $p\in\overline{D}$ in $X_B$.
Now let
$E=\bigcup_{s\in D}\{s\upharpoonright j:j\le\dom(s)\}\subseteq A^-$.
Also the set $E$ is an infinite subtree of ${}^{<\omega}2$,
hence it must have an infinite branch -- say,
$\{g\upharpoonright j:j\in\omega\}$ for $g\in{}^{\omega}2$.
The procedure for obtaining $E$ from $A^-$ guarantees that
$g\notin B$.
As $\{g\upharpoonright j:j\in\omega\}\subseteq A^-$,
it follows from the definition of $A^-$ that
$\{j\in\omega:g\upharpoonright j\in A\}$ is infinite.
\end{proof}

\begin{lemma}
\label{nwd}

Let $R\subseteq B$ be such that,
for every $g\in({}^{\omega}2)\setminus B$,
there is $j\in\omega$ such that
no $f\in R$ extends $g\upharpoonright j$.
Then $R$ is nowhere dense in ${}^{\omega}2$.

\end{lemma}

\begin{proof}

Fix, for each $g\in({}^{\omega}2)\setminus B$,
a $j_g\in\omega$ such that
$\forall f\in R\;(g\upharpoonright j_g\nsubseteq f)$.
Then
$$
\bigcup_{g\in({}^{\omega}2)\setminus B}
V_{g\upharpoonright j_g}\cap{}^{\omega}2
$$
is an open subset of ${}^{\omega}2$ which, furthermore,
is dense in ${}^{\omega}2$ --
since $({}^{\omega}2)\setminus B$ is dense in ${}^{\omega}2$.
It follows from the choice of the $j_g$s that
$R$ is disjoint from this dense open set.
Thus, $R$ is nowhere dense in ${}^{\omega}2$.
\end{proof}

\begin{defin}
\label{def.simple}

Let $X$ be a topological space and $x\in X$.
A family $\mathcal{C}$ of nonempty open subsets of $X$ is
\emph{simple at $x$}
if
every $A\subseteq X$ with $x\in\overline{A}$ includes a finite set
that intersects every element of $\mathcal{C}$.

\end{defin}

The following result parallels Theorems 2.10 and 2.11 of
\cite{barman-dow2},
and is inspired by Theorem 1 of \cite{sch95}.

\begin{prop}
\label{equiv.simple}

Let $X$ be a topological space and $x\in X$.
Consider the following statements:
\begin{itemize}
\item[$(a)$]
every local base for $X$ at $x$
is a countable union of simple families;
\item[$(b)$]
$\tau_x$ is a countable union of simple families;
\item[$(c)$]
there is a local base for $X$ at $x$
that is a countable union of simple families;
\item[$(d)$]
$\2\uparrow\gfin(\Omega_x,\Omega_x)$ on $X$.
\end{itemize}
Then $(a)\Leftrightarrow(b)\Leftrightarrow(c)\Rightarrow(d)$.
Moreover,
if $X$ is
countable,
then the four statements are equivalent.

\end{prop}

\begin{proof}

We prove the second part only --
namely, the implication $(d)\Rightarrow(b)$.
Let then $X$ be a countable space on which there is
a winning strategy $\varphi$ for \2 in the game
$\gfin(\Omega_x,\Omega_x)$.

Define
$$\mathcal{U}_{\emptyset}=\bigcap_{A\in\Omega_x}\{U\in\tau_x:U\cap\varphi(A)\neq\emptyset\}.$$
As $X$ is countable, the set
$\{\varphi(A):A\in\Omega_x\}\subseteq[X]^{<\aleph_0}$ is countable;
let $\mathcal{A}_{\emptyset}\subseteq\Omega_x$ be such that
$\{\varphi(A):A\in\Omega_x\}=\{\varphi(A^{\emptyset}_i):i\in\omega\}$.
Now, for each $(i_0)\in {}^1 \omega$, let
$$\mathcal{U}_{(i_0)}=\bigcap_{A\in\Omega_x}\{U\in\tau_x:U\cap\varphi(A^{\emptyset}_{i_0},A)\neq\emptyset\}.$$
Let $\{A^{(i_0)}_i:i\in\omega\}\subseteq\Omega_x$ be such that
$\{\varphi(A^{\emptyset}_{i_0},A):A\in\Omega_x\}=\{\varphi(A^{\emptyset}_{i_0},A^{(i_0)}_i):i\in\omega\}$.
For each $(i_0,i_1)\in {}^2 \omega$, define
$$\mathcal{U}_{(i_0,i_1)}=\bigcap_{A\in\Omega_x}\{U\in\tau_x:U\cap\varphi(A^{\emptyset}_{i_0},A^{(i_0)}_{i_1},A)\neq\emptyset\},$$
and pick $\{A^{(i_0,i_1)}_i:i\in\omega\}\subseteq\Omega_x$ satisfying
$\{\varphi(A^{\emptyset}_{i_0},A^{(i_0)}_{i_1},A):A\in\Omega_x\}=\{\varphi(A^{\emptyset}_{i_0},A^{(i_0)}_{i_1},A^{(i_0,i_1)}_i):i\in\omega\}$.
Then define
$$\mathcal{U}_{(i_0,i_1,i_2)}=\bigcap_{A\in\Omega_x}\{U\in\tau_x:U\cap\varphi(A^{\emptyset}_{i_0},A^{(i_0)}_{i_1},A^{(i_0,i_1)}_{i_2},A)\neq\emptyset\}$$
for each $(i_0,i_1,i_2)\in {}^3 \omega$, and so on.
By proceeding in this fashion, we construct $\mathcal{U}_s$ for every
$s\in {}^{<\omega}\omega$, by induction on $\dom(s)$.

Note that each family $\mathcal{U}_s$ is simple.
The proof will be finished once we show that
$\bigcup_{s\in {}^{<\omega} \omega}\mathcal{U}_s=\tau_x$.

Suppose, by way of contradiction, that there is
$V\in\tau_x\setminus\bigcup_{s\in {}^{<\omega} \omega}\mathcal{U}_s$.
We can then recursively pick $i_0,i_1,i_2,\dots$ in $\omega$ such that
$$V\cap\varphi(A^{\emptyset}_{i_0},A^{(i_0)}_{i_1},\dots,A^{(i_0,\dots,i_{n-1})}_{i_n})=\emptyset$$
for each $n\in\omega$.
But then
$$(\varphi(A^{\emptyset}_{i_0}),\varphi(A^{\emptyset}_{i_0},A^{(i_0)}_{i_1}),\dots,\varphi(A^{\emptyset}_{i_0},A^{(i_0)}_{i_1},\dots,A^{(i_0,\dots,i_{n-1})}_{i_n}),\dots)$$
is the sequence of \2's moves in a play of $\gfin(\Omega_x,\Omega_x)$
in which \2 makes use of the strategy $\varphi$ and loses.
This contradicts the fact that $\varphi$ is a winning strategy.
\end{proof}

We can now proceed to the proof that
$\2\not\,\uparrow\gfin(\Omega_p,\Omega_p)$ on $X_B$.

We will make use of Proposition \ref{equiv.simple}.
Suppose, towards a contradiction, that
$\{U_H:H\in[B]^{<\aleph_0}\}
=\bigcup_{n\in\omega}\mathcal{F}_n$,
where $\mathcal{F}_n$ is simple for each $n\in\omega$.

Let $n\in\omega$ be arbitrary.
Define $\mathcal{H}_n=
\{H\in[B]^{<\aleph_0}:U_H\in\mathcal{F}_n\}$.
As $\mathcal{F}_n$ is simple,
it follows from Lemma \ref{closure.XB} that
every $A\subseteq{}^{<\omega}2$ satisfying
$|A\cap\{g\upharpoonright j:j\in\omega\}|\ge\aleph_0$
for some $g\in({}^{\omega}2)\setminus B$
includes a finite subset $A_0$
that meets every element of $\mathcal{F}_n$.
By considering the particular case
$A=\{g\upharpoonright j:j\in\omega\}$ of this statement,
we obtain that, for every $g\in({}^{\omega}2)\setminus B$,
there is a finite $J_g\subseteq\omega$ such that,
for every $H\in\mathcal{H}_n$, there is some $j'\in J_g$
satisfying
$\forall h\in H\;(g\upharpoonright j'\nsubseteq h)$.
Therefore, for every $g\in({}^{\omega}2)\setminus B$,
it follows that $j_g=\max J_g$ satisfies
$g\upharpoonright j_g\nsubseteq h$
for all $h\in\bigcup\mathcal{H}_n$.
Hence, $\bigcup\mathcal{H}_n$ is nowhere dense in ${}^{\omega}2$
by Lemma \ref{nwd}.

Now, as $B\subseteq{}^{\omega}2$ is not meagre
(see e.g. \cite[pp. 29--30]{oxtoby}),
there is
$x\in B\setminus\bigcup_{n\in\omega}\bigcup\mathcal{H}_n$.
The set $U_{\{x\}}$ must be an element of $\mathcal{F}_n$
for some $n\in\omega$
-- which yields $\{x\}\in\mathcal{H}_n$,
thus contradicting the choice of $x$.
\end{proof}

The results obtained so far can be summarized by the following diagram:

\bigskip
\begin{tikzpicture}
  \matrix (m) [matrix of nodes, row sep=20mm, column sep=5mm]
{
$\2\uparrow\gn{n}(\Omega_x \Omega_x)$ & $\2\uparrow\gn{n + 1}(\Omega_x \Omega_x)$ & $\2\uparrow\gn{f}(\Omega_x \Omega_x)$ & $\2\uparrow\gfin(\Omega_x \Omega_x)$\\
$\1\not\,\uparrow\gn{n}(\Omega_x \Omega_x)$ & $\1\not\,\uparrow\gn{n + 1}(\Omega_x \Omega_x)$ & $\1\not\,\uparrow\gn{f}(\Omega_x \Omega_x)$ & $\1\not\,\uparrow\gfin(\Omega_x \Omega_x)$\\
& $\sone(\Omega_x, \Omega_x)$ & $\sn{f}(\Omega_x, \Omega_x)$ & $\sfin(\Omega_x, \Omega_x)$\\ 
};

\draw[->] (m-1-1) to node[auto]{\ref{ex.k.k1}} (m-1-2);
\draw[->] (m-1-2) to node[auto]{\ref{ex.salvador}} (m-1-3);
\draw[->] (m-1-3) to node[auto]{\ref{CpR}} (m-1-4);

\draw[->] (m-2-1) to node[auto]{\ref{ex.k.k1}} (m-2-2);
\draw[->] (m-2-2) to node[auto]{\ref{ex.salvador}} (m-2-3);
\draw[->] (m-2-3) to node[auto]{\ref{CpR}} (m-2-4);

\draw[->] (m-3-2) to node[auto]{\ref{ex.salvador}} (m-3-3);
\draw[->] (m-3-3) to node[auto]{\ref{CpR}} (m-3-4);

\draw[->] (m-1-1) to node[auto]{\ref{ex.gruen}} (m-2-1);
\draw[->] (m-1-2) to node[auto]{\ref{ex.gruen}} (m-2-2);
\draw[->] (m-1-3) to node[auto]{\ref{ex.gruen}} (m-2-3);
\draw[->] (m-1-4) to node[auto]{\ref{ex.gruen}} (m-2-4);

\draw[->] (m-2-2) to node[auto]{\ref{ex.scheepers}} (m-3-2);
\draw[->] (m-2-3) to node[auto]{\ref{ex.scheepers}} (m-3-3);
\draw[->] (m-2-4) to node[auto]{\ref{ex.scheepers}} (m-3-4);

\end{tikzpicture}

\bigskip
In the diagram, the arrows maked with \ref{ex.scheepers} indicate that
Example \ref{ex.scheepers} shows that the converse of the
corresponding implication does not hold; and so forth.

\section{New directions and open problems}

The selective topological properties we considered in this work are
the ``countable tightness'' particular cases of a broader
(non-topological) framework introduced by M. Scheepers in \cite{sch1},
which gave rise to a field of research that today is known as the
study of \emph{selection principles}.

\begin{defin}[Scheepers \cite{sch1}]
\label{def.S}
Let $\mA,\mB$ be nonempty families of nonempty sets.
$\sone(\mA,\mB)$ and $\sfin(\mA,\mB)$ denote, respectively, the
following statements:
\begin{itemize}
\item[$\sone(\mA,\mB)$ $\equiv$]
for every sequence $(A_n)_{n\in\omega}$ of elements of $\mA$, there is
a sequence $(b_n)_{n\in\omega}$ such that $b_n\in A_n$ for each
$n\in\omega$ and $\{b_n:n\in\omega\}\in\mB$;
\item[$\sfin(\mA,\mB)$ $\equiv$]
for every sequence $(A_n)_{n\in\omega}$ of elements of $\mA$, there is
a sequence $(F_n)_{n\in\omega}$ of finite sets such that
$F_n\subseteq A_n$ for each $n\in\omega$ and
$\bigcup_{n\in\omega}F_n\in\mB$.
\end{itemize}
\end{defin}

Although in this work we have concentrated on the instance
$(\mA,\mB)=(\Omega_x,\Omega_x)$ of the property schemas above,
there are various contexts in which properties of this kind arise
naturally.
A typical example is provided by the classical Rothberger \cite{roth}
and Menger \cite{hur} covering properties, which can be expressed in
terms of selection principles as $\sone(\mO,\mO)$ and $\sfin(\mO,\mO)$
respectively; here, $\mO$ stands for the family of all of the open
covers of a given topological space.

Some other selection principles have been added to $\sone$ and $\sfin$
in the literature, such as the selection principle $\sn{f}$ from
Section 3 -- which has appeared in e.g. the following result, proved
in Lemma 3.12 of \cite{salvador} and Lemma 3 of \cite{cie}
(see also \cite[Appendix A]{tsaban}):

\begin{prop}[Garc\'ia-Ferreira--Tamariz-Mascar\'ua \cite{salvador},
    Bukovsk\'y--Ciesielski \cite{cie}]
\label{roth.Gf}
Let $f\in{}^{\omega}\NN$ be arbitrary. Then the properties
$\sone(\mO,\mO)$ and $\sn{f}(\mO,\mO)$ are equivalent.
\end{prop}

Each selection principle has game naturally associated to it.
For example, in the original definition due to Scheepers, we have:

\begin{defin}[Scheepers \cite{sch3}]
\label{def.G}
Let $\mA,\mB$ be nonempty families of nonempty sets.
$\gone(\mA,\mB)$ and $\gfin(\mA,\mB)$ denote, respectively, the
following games.
\begin{itemize}
\item[$\cdot$]
In each inning $n\in\omega$ of $\gone(\mA,\mB)$,
\1 chooses $A_n\in\mA$, and then \2 picks $b_n\in A_n$.
The winner is \2 if $\{b_n:n\in\omega\}\in\mB$,
and \1 otherwise.
\item[$\cdot$]
In each inning $n\in\omega$ of $\gfin(\mA,\mB)$,
\1 chooses $A_n\in\mA$,
and then \2 picks a finite subset $F_n\subseteq A_n$.
The winner is \2 if $\bigcup_{n\in\omega}F_n\in\mB$,
and \1 otherwise.
\end{itemize}
\end{defin}

Though the selection principles $\sn{f}$ and $\sn{k}$ have already
occurred in the literature,
to our knowledge
their game versions $\gn{f}$ and $\gn{k}$ have not been object
of study thus far.
The fact that such counterintuitive differences between games such as
$\gn{k}$ and $\gn{k+1}$
spring in a typically infinitary context such as topological spaces
-- in which many general properties are not sensitive to changes from
a finite number to another --
suggests that investigating other instances of these games might lead
to a whole new class of interesting results.

\begin{prob}
What can be said about the relation between the various games
$\gn{k}(\mA,\mB)$, $\gn{f}(\mA,\mB)$ and $\gfin(\mA,\mB)$
-- and their associated selective properties --
for other pairs $(\mA,\mB)$?
\end{prob}

As we have seen in this work, the games of form
$\gn{k}(\Omega_x,\Omega_x)$ bear differences between them that are not
verified in their selective counterparts $\sn{k}(\Omega_x,\Omega_x)$.
Since, by Proposition \ref{roth.Gf}, the class of mutually equivalent
selective properties of form $\sn{\square}(\mO,\mO)$ is even greater,
one may wonder whether the differences between the games are also
verified in this case.

\begin{prob}
\label{(O,O)?}
For $k\in\NN$, is there a space on which the games
$\gn{k}(\mO,\mO)$ and $\gn{k+1}(\mO,\mO)$ are not equivalent?
What about their relation to the game $\gn{f}(\mO,\mO)$
for $f\in{}^{\omega}\NN$?
\end{prob}

\section*{Acknowledgements}

We wish to express our gratitude to Boaz Tsaban
for bringing the paper \cite{salvador} to our attention
and for suggesting Proposition \ref{equiv-games}.

\end{document}